\documentclass[12pt]{article}
\usepackage{amsfonts,amssymb,amsmath,titling,titlesec,cancel,amsthm,color, graphicx,hyperref}
\usepackage{epstopdf}

\AppendGraphicsExtensions{.gif}

\titleformat{\section}{\large\bfseries}{\thesection.}{1em}{}

\newtheorem{theorem}{Theorem}[section]
\newtheorem{lemma}[theorem]{Lemma}

\begin{document}

\title{Degree of Regularity of Linear Homogeneous Equations}
\author{Kavish Gandhi\thanks{MIT-PRIMES, Department of Mathematics, Massachusetts Institute of Technology.}, Noah Golowich$^*$ and L\'aszl\'o Mikl\'os Lov\'asz\thanks{Department of Mathematics,
    Massachusetts Institute of Technology,
    Cambridge, MA 02139-4307.
    Email: {\tt lmlovasz@math.mit.edu}.}}
\date{\today}
\maketitle

\begin{abstract}
We define a linear homogeneous equation to be \emph{strongly $r$-regular} if, when a finite number of inequalities is added to the equation, the system of the equation and inequalities is still $r$-regular.  In this paper, we show that, if a linear homogeneous equation is $r$-regular, then it is strongly $r$-regular. In 2009, Alexeev and Tsimerman introduced a family of equations, each of which is $(n-1)$-regular but not $n$-regular, verifying a conjecture of Rado from 1933. These equations are actually strongly $(n-1)$-regular as an immediate corollary of our results.  
\end{abstract}

\section{Introduction}
\label{sec:intro}
In 1927, van der Waerden \cite{waerden_beweis_1927} proved his seminal theorem stating that, for any finite coloring of the positive integers, there always exists a monochromatic arithmetic progression of arbitrary length. Subsequently, in 1933, Rado \cite{rado_studien_1933} expanded on this theorem, finding a necessary and sufficient condition for the partition regularity of systems of linear homogeneous equations. In the case of a single linear homogeneous equation of the form
\begin{equation}
\label{eq:initialeq}
a_1x_1 + a_2x_2 +a_3x_3 + \cdots + a_nx_n = 0 \quad : \quad a_i \ne 0, \ a_i \in \mathbb{Z},
\end{equation}
given any finite coloring of the integers, Rado proved that there exist positive monochromatic integers $(x_1, x_2, \ldots, x_n)$ that satisfy the equation if and only if a nonempty subset of $\{a_1, a_2, \ldots, a_n\}$ sums to 0. An equation for which there exists a monochromatic solution given any finite coloring is defined as {\it regular}.

Not all linear homogeneous equations are regular, however. Those that are not regular are classified as follows: given a positive integer $r$, a linear homogeneous equation is called {\it $r$-regular} if, for every coloring of the positive integers with $r$ colors, there always exists a monochromatic solution $x_1, x_2, \ldots, x_n$ to the equation. The {\it degree of regularity} of a linear homogeneous equation is defined as the largest positive integer $r$ such that the equation is $r$-regular.

Rado conjectured in 1933 \cite{rado_studien_1933} that for every positive integer $n$, there exists a linear homogeneous equation with degree of regularity equal to $n$. This conjecture was open for a long time until it was proven in 2009 by Alexeev and Tsimerman \cite{alexeev_equations_2010}. Specifically, they proved that for each $n$ the equation
\begin{equation}
\label{eq:alexeev}
 \left(1 - \displaystyle\sum\limits_{i = 1}^{n-1} \frac{2^i}{2^i - 1}\right)x_1 + \displaystyle\sum\limits_{i = 1}^{n-1} \frac{2^i}{2^i-1}x_{i+1}  = 0
\end{equation}
is $(n-1)$-regular but not $n$-regular.  To show that these equations are $(n-1)$-regular, they noted that there must be an $i$, $0 < i < n$, such that $x$ and $2^i x$ are the same color. Otherwise, the $n$ integers $x, 2x, 4x, \ldots 2^{n-1}x$ would all be  different colors, which is impossible in an $(n-1)$-coloring. Alexeev and Tsimerman then noted that the following is a monochromatic solution to $(\ref{eq:alexeev})$:
\begin{eqnarray}
x_1 = x_2 = \cdots = x_{i} = x_{i+2} = \cdots = x_n &=& 2^ix\nonumber\\
x_{i+1} &=& x\nonumber.
\end{eqnarray}

However, this solution is degenerate in that all but one of $x_1, \ldots, x_n$ are equal to one another and are thus guaranteed to be the same color for any coloring of the positive integers.  This observation raises the question: can we still find a monochromatic solution to the equation $(\ref{eq:alexeev})$ if we restrict the $x_1, \ldots, x_n$ such that they cannot be equal?  More generally, if we add a finite number $k$ of inequalities of the form 
\begin{equation}
\label{eq:initialineq}
A_{j,1}x_1 + \cdots + A_{j,n}x_n \ne 0 \quad : \quad 1 \le j \le k, \mbox{  } A_{j,1}, \ldots, A_{j,n} \in \mathbb{Z},
\end{equation}
is the degree of regularity affected? We call an equation {\it strongly $r$-regular} if, when any inequalities of the form $(\ref{eq:initialineq})$ that are not multiples of the equation are specified, the system of the equation and the inequalities is $r$-regular. We note that none of these inequalities can be a multiple of the original equation, as then the system has no solution. Adding inequalities is equivalent to deleting a finite number of hyperplanes from the set of possible solutions.


In this paper, we show that if an equation is $r$-regular, then it is strongly $r$-regular. Therefore, the equation $(\ref{eq:alexeev})$ is strongly $(n-1)$-regular, answering the question discussed above. It follows from our main theorem as another corollary that a set of equations Fox and Radoi\v ci\'c proved to be $2$-regular but not 3-regular \cite{fox_axiom_2005} is strongly 2-regular. Thus, the degree of regularity of these equations is also not affected by the addition of inequalities. 

The same question of adding inequalities in the form of $(\ref{eq:initialineq})$ and determining whether partition regularity is affected can be asked for regular equations. In 1998, Hindman and Leader \cite{hindman_partition_1998} resolved this question. It follows easily from Theorem 2 of \cite{hindman_partition_1998} that for any regular linear homogeneous equation, if any finite set of inequalities $(\ref{eq:initialineq})$ is specified, the system of the equation and inequalities is still regular.

The organization of this paper is as follows: In Section $\ref{sec:mainthm}$ we prove an important lemma and then our main result. In Section $\ref{sec:implications}$, we discuss some implications and extensions of our result.

\section{Proof of strong $r$-regularity}
\label{sec:mainthm}
The lemma below is very similar to Theorem 8 in Chapter 3 of Graham, Rothschild, \& Spencer \cite{graham_ramsey_1990}, and the proof is almost identical. We present it here for completeness. This lemma allows us to find a set of monochromatic progressions $P_1, P_2,\ldots, P_n$ of the same color, such that there are many choices of $x_i \in P_i$ with $(x_1,x_2,\ldots,x_n)$ a monochromatic solution to a linear homogeneous equation.
\begin{lemma} 
\label{thm:imptlemma}
Let 
\begin{equation}
\label{eq:lheqinlemma}
a_1x_1 + \cdots + a_nx_n = 0
\end{equation}
be a linear homogeneous equation which is $r$-regular and $C$ be a positive constant.  Then, if $\mathbb{N}$ is finitely colored with $r$ colors, there exist $x_1, \ldots, x_n$ satisfying $(\ref{eq:lheqinlemma})$ and $d>0$ so that all
\begin{equation}
\label{eq:monoseq}
x_i + \lambda d \quad : \quad 1\le i \le n, |\lambda| \le C
\end{equation}
are the same color.
\end{lemma}

\begin{proof} 
By the Compactness principle, we can find a constant $R$ such that for any $r$-coloring of $[1, R]$, there exists a monochromatic solution to $(\ref{eq:lheqinlemma})$.  Now, let $\omega$ be an $r$-coloring of $\mathbb{N}$.  We define an $r^{R}$ coloring on $\mathbb{N}$, namely $\omega'$, by 
\[\omega'(\alpha) = \omega'(\beta) \ \ \ \ \text{iff} \ \ \ \ \omega(\alpha i) = \omega(\beta i) \mbox{ for } 1 \le i \le R. \]
We define $K = CR^{n-1}$. Within our coloring $\omega'$, we find a monochromatic arithmetic progression $P$ of length $2K + 1$ of the form 
\begin{equation}
\label{eq:othermonoseq}
a + \lambda d \quad : \quad |\lambda| \le K.
\end{equation}
Note that for the original coloring $\omega$, this means that $P$, $2\cdot P$,\ldots, $R\cdot P$ are all monochromatic.  Since $(\ref{eq:lheqinlemma})$ is homogeneous, the $r$-coloring $\omega$ of $\{a, 2a, \ldots, Ra\}$ yields a monochromatic solution, $ay_1, \ldots, ay_n$.  

Now we set $x_i = ay_i$ for $1 \le i \le n$. Moreover, we let $y$ be the least common multiple of $y_1, \ldots, y_n$ and $d' = d y$. Therefore,
\[x_i + \lambda d'  = ay_i + \lambda d y = y_i\left(a + \lambda d\left(\frac{y}{y_i}\right)\right).\]
Notice that $\lambda \le C$ and $\frac{y}{y_i} \le y_1\ldots y_{i-1} y_{i+1} \ldots y_n \le R^{n-1}$, so $\lambda y/ y_i \le K$.  

Thus, $a + \lambda d y/y_i$ belongs to $P$, which implies that 
\[\omega'\left(a + \lambda d \left(\frac{y}{y_i}\right)\right) = \omega'(a),\]
and thus, by our definition of $\omega'$,
\begin{equation}
\omega(x_i + \lambda d') = \omega\left(y_i(a+\lambda d\left(\frac{y}{y_i}\right))\right) =  \omega(ay_i) = \omega(x_i).\nonumber
\end{equation}
Since $\omega(x_i)$ is constant for $1 \leq i \leq n$, $\omega(x_i + \lambda d')$ is constant for $1 \leq i \leq n$ and $|\lambda| \leq C$.
\end{proof}

The below theorem is our main result. It uses Lemma $\ref{thm:imptlemma}$ to prove that all $r$-regular equations are strongly $r$-regular.

\begin{theorem}
\label{thm:major}
Assume that the equation
\begin{equation}
\label{eq:lheqn}
a_1x_1 + a_2x_2 + \cdots + a_nx_n = 0
\end{equation}
is $r$-regular. Then it is strongly $r$-regular.
\end{theorem}
\begin{proof}
Let
\begin{equation}
\label{eq:ineqs}
A_{j,1}x_1 + A_{j,2}x_2 + \cdots + A_{j,n}x_n \neq 0 \quad : \quad A_{j,i} \in \mathbb{Z}, 1 \leq j \leq k
\end{equation}
be any collection of $k$ inequalities, none of which are multiples of $(\ref{eq:lheqn})$.
Let $\omega : \mathbb{N} \rightarrow \{1, 2, \ldots, r\}$ be any $r$-coloring of the positive integers. We pick a positive integer $C$, to be specified later. By Lemma $\ref{thm:imptlemma}$, there exist $x_1, \ldots, x_n, d > 0$ such that $(x_1, \ldots, x_n)$ satisfy $(\ref{eq:lheqn})$ and so that, for $1 \leq i \leq n$, and $|\lambda| \leq C$, $\omega(x_i + \lambda d)$ is constant.

We now claim that there exists a sequence $(\lambda_1, \ldots, \lambda_n)$ such that $(x_1 + \lambda_1d, \ldots, x_n + \lambda_nd)$ satisfies both $(\ref{eq:lheqn})$ and $(\ref{eq:ineqs})$. To show this, note that we need
\begin{eqnarray}
\sum_{i = 1}^n a_i(x_i + \lambda_id) = 0\nonumber,
\end{eqnarray}
or, equivalently,
\begin{eqnarray}
\label{eq:necforeq}
\sum_{i = 1}^n a_i\lambda_i = 0,
\end{eqnarray}
since $\sum_{i = 1}^n a_ix_i = 0$. We also need, for each $j$ where $1 \leq j \leq k$,
\begin{equation}
\sum_{i = 1}^n A_{j,i}(x_i + \lambda_id) \neq 0\nonumber,
\end{equation}
which is equivalent to
\begin{equation}
\label{eq:necforineqs}
\sum_{i = 1}^n A_{j,i}\lambda_i \neq - \frac{\sum_{i = 1}^n A_{j,i}x_i}{d}.
\end{equation}
If, for any $j$, the coefficients of $\lambda_i$ in $(\ref{eq:necforineqs})$ are each a multiple of the corresponding coefficients of $\lambda_i$ in $(\ref{eq:necforeq})$, then the $j^{th}$ inequality of $(\ref{eq:ineqs})$ would be a multiple of $(\ref{eq:lheqn})$, a contradiction. Now, we use the fact that if we have a linear equation with rational coefficients, and a finite set of linear inequalities such that no inequality is a multiple of the original equation, then we can find a solution that satisfies the equation and the inequalities. This gives us the needed set $(\lambda_1, \lambda_2, \ldots, \lambda_n)$. Notice that we may choose 
\begin{equation}
C = \left\lceil{\frac{(k+1)(|a_1| + \cdots + |a_n|)}{2}}\right\rceil,\nonumber
\end{equation} 
since it is possible to choose each of of $\lambda_1, \ldots, \lambda_{n-1}$ out of the set $\{a_n, 2a_n, \ldots (k+1)a_n\}$, meaning that $|\lambda_1|, \ldots, |\lambda_{n-1}| \leq (k+1)a_n$. Moreover,
\begin{equation}
|\lambda_n| = \left|\frac{a_1\lambda_1 + \cdots + a_n\lambda_n}{a_n}\right|\nonumber
\end{equation} is at most $(k+1)(|a_1| + \cdots + |a_{n-1}|)$. Therefore, twice the value of $C$ given above is greater than the absolute value of each of $\lambda_1, \ldots, \lambda_n$. This means that the $n$-tuplet $(x_1 + \lambda_1d, \ldots, x_n + \lambda_nd)$ is monochromatic and satisfies $(\ref{eq:lheqn})$ and $(\ref{eq:ineqs})$ for all $j$ values.
\end{proof}

\section{Implications}
\label{sec:implications}
As a result of our theorem, we resolve our motivational question: we can find monochromatic solutions to Alexeev and Tsimermaan's family of equations that are distinct, and more generally, that additionally satisfy a finite set of inequalities.  We can also generalize Alexeev and Tsimerman's method of proving that an equation is $r$-regular to a much larger class of equations. To do so, we make the following definitions:

We define an upper triangular $m$ by $m$ matrix
\begin{equation}\nonumber
C = \left[\begin{array}{ccccc}
c_{1,1}&c_{1,2}&c_{1,3}&\cdots&c_{1,m}\\
0&c_{2,2}&c_{2,3}&\cdots&c_{2,m}\\
0&0&c_{3,3}&\cdots&c_{3,m}\\
\vdots&\vdots&\vdots&\ddots&\vdots\\
0&0&0&\cdots&c_{m,m}
\end{array}\right]
\end{equation}
to have the {\it linkage property} if, for each integer $i$ where $1 \leq i \leq m-1$, the following holds for all $j$ where $i < j \leq m$:
\begin{equation}
c_{1,i}\cdot c_{i+1,j} = c_{1,j}\nonumber.
\end{equation}
We will only be considering matrices with the linkage property that have all elements positive.

Given an equation $(\ref{eq:lheqn})$, we let, for $1 \leq l \leq n$, 
\begin{eqnarray}
S_l &=& -\frac{\left(\sum_{i = 1}^{n} a_i\right) - a_l}{a_l}\nonumber.
\end{eqnarray}
Notice that for any coloring of the integers that excludes monochromatic solutions to a linear homogeneous equation, if $x$ is an integer, then $x$ and $S_lx$ cannot have the same color. Therefore, the number $S_l$ is otherwise known as a forbidden ratio\cite{alexeev_minimal_2007}.

We now have the following:
\begin{theorem}
\label{thm:mainthm2}
Assume that, given an equation $(\ref{eq:lheqn})$ and inequalities $(\ref{eq:ineqs})$, there exists an upper triangular $m$ by $m$ matrix with the linkage property where each nonzero entry is positive and equal to $S_l$ for some $l$ where $1 \leq l \leq n$. Then the system of the equation and the inequalities is strongly $m$-regular.
\end{theorem}
\begin{proof}
By Theorem $\ref{thm:major}$, it is sufficient to show that the linear homogeneous equation is $m$-regular. To show this, we assume the contrary; that is, there exists some $m$-coloring $c : \mathbb{N} \rightarrow \{1, 2, \ldots, m\}$ of the natural numbers which excludes monochromatic solutions. 

We assume that the matrix:
\begin{equation}
S = \left[\begin{array}{ccccc}
S_{l_{1,1}}&S_{l_{1,2}}&S_{l_{1,3}}&\cdots&S_{l_{1,m}}\\
0&S_{l_{2,2}}&S_{l_{2,3}}&\cdots&S_{l_{2,m}}\\
0&0&S_{l_{3,3}}&\cdots&S_{l_{3,m}}\\
\vdots&\vdots&\vdots&\ddots&\vdots\\
0&0&0&\cdots&S_{l_{m,m}}
\end{array}\right]\nonumber
\end{equation}
is filled with positive $S_l$ $(1 \leq l \leq n)$ in the non-0 slots and has the linkage property. 

Let $x$ be a positive integer which has the property that for $1 \leq i \leq j \leq m$, $S_{l_{i,j}}x$ is also a positive integer. Without loss of generality, let $c(x) = 1$. By the forbidden ratios $S_{l_{1,1}}, S_{l_{1,2}}, \ldots, S_{l_{1,m}}$, for $1 \leq j \leq m$, $1 = c(x) \neq c(S_{l_{1,j}}x)$. Without loss of generality, let $c(S_{l_{1,1}}x) = 2$. By the forbidden ratios, $S_{l_{2,2}}, \ldots, S_{l_{2,m}}$, and since $S$ has the linkage property, for $2 \leq j \leq m$, $c(S_{l_{1,j}}x) \neq c(S_{l_{1,1}}x) = 2$. Therefore, for $2 \leq j \leq m$, $c(S_{l_{1,j}}x) \in \{3, \ldots, m\}$.

We now proceed by induction, assuming that for some integer $t$ where $m \geq t \geq 3$, for $t-1 \leq j \leq m$, $c(S_{l_{1,j}}x) \in \{t, \ldots, m\}$. Without loss of generality, let $c(S_{l_{1,t-1}}x) = t$. By the forbidden ratios $S_{l_{t,t}}, \ldots, S_{l_{t,m}}$, and since $S$ has the linkage property, for $t \leq j \leq m$, $c(S_{l_{1,j}}x) \neq c(S_{l_{1, t-1}}x) = t$. Therefore, for $t \leq j \leq m$, $c(S_{l_{1,j}}x) \in \{t + 1, \ldots, m\}$.

Eventually, when we reach $t = m$, we get that $c(S_{l_{1,m}}x) \in \emptyset$. This is a contradiction, since we originally assumed that $c$ assigns every positive integer a color.
\end{proof}

Investigating the degree of regularity of any non-regular equation is an interesting direction of future research. The above theorem begins to answer this question, but there are many equations which do not satisfy the preconditions of the theorem. Moreover, there are instances in which the linkage property can show that an equation is $m$-regular but where the degree of regularity of an equation is greater than $m$. Therefore, it would be desirable to extend the above method of proof to prove more about the degree of regularity of certain equations as well as give further insight into when an equation is $r$-regular.  

\section{Acknowledgments}

We would like to thank Professor Jacob Fox for suggesting the project and for helpful discussions concerning the paper.  We would also like to thank Dr. Tanya Khovanova for her useful comments and suggestions on the paper.  Finally, we would like to acknowledge the MIT-PRIMES program for giving us the opportunity to perform this research.

\bibliographystyle{ieeetr}

\end{document}